\documentclass{amsart}

\usepackage{amssymb,amsmath}
\usepackage{cite}
\usepackage{graphicx}
\usepackage{color}

\theoremstyle{plain}
\newtheorem{theorem}{Theorem}

\newtheorem{lemma}[theorem]{Lemma}

\theoremstyle{definition}

\newtheorem*{ack}{Acknowledgement}
\newtheorem*{rem}{Remark}

\newcommand{\Xp}{\vphantom{X}^f\!X} 
\newcommand{\tXp}{\vphantom{X}^f\!\tilde X} 
\newcommand{\Yp}{\vphantom{Y}^f\!Y} 
\newcommand{\tYp}{\vphantom{Y}^f\!\tilde Y} 
\newcommand{\Zp}{\vphantom{Z}^f\!Z} 
\newcommand{\Sp}{\vphantom{S}^f\!S} 
\newcommand{\tSp}{\vphantom{S}^f\!\tilde S} 
\newcommand{\Tp}{T^f} 

\begin{document}

\title[Asymptotic properties of delayed Brownian motion]{Asymptotic properties of Brownian motion delayed by inverse subordinators}

\author{Marcin Magdziarz}
\address{Hugo Steinhaus Center, Institute of Mathematics and Computer Science,
Wroclaw University of Technology, Wyspianskiego 27, 50-370 Wroclaw, Poland}
\email{marcin.magdziarz@pwr.wroc.pl}

\author{Ren\'e L. Schilling}
\address{Technische Universit\"at Dresden, Institut f\"{u}r Mathematische Stochastik, 01062 Dresden, Germany}
\email{rene.schilling@tu-dresden.de}

\subjclass[2010]{Primary 60G17; Secondary 60G52}

\begin{abstract}
We study the asymptotic behaviour of the time-changed stochastic process
$\Xp(t)=B(\Sp (t))$, where $B$ is a standard one-dimensional Brownian motion
and $\Sp$ is the (generalized) inverse of a subordinator, i.e.\ the first-passage time process
corresponding to an increasing L\'evy process with Laplace exponent $f$.
This type of processes plays an important role in statistical physics in the modeling of
anomalous subdiffusive dynamics.
The main result of the paper is the proof of the mixing property
for the sequence of stationary increments of a subdiffusion process.
We also investigate various martingale properties,
derive a generalized Feynman-Kac formula, the laws of large numbers and of the iterated logarithm
for $\Xp$.
\end{abstract}

\maketitle

\section{Introduction}

Let $(\Omega,\mathcal{F},\mathbb{P})$ be some probability space. A subordinator $\Tp(t)$, $t\geq 0$, is an increasing L\'evy process, i.e.\ a stochastic process with stationary and independent increments whose sample paths are right-continuous with finite left limits. Being a Markov process, the law of $\Tp$ is uniquely characterized by the one-step transition functions. By definition, these are one-sided, infinitely divisible and the Laplace transform of $\Tp(t)$ can be written as, cf.~\cite{Sato},
\[
    \mathbb{E}\left( e^{-u\Tp(t)}   \right)=e^{-t f(u)},
\]
where the characteristic (Laplace) exponent $f(u)$ is a Bernstein function; it is well known, cf.\ \cite{Rene_book}, that all characteristic exponents are of the form
\[
    f(u)=\lambda u+\int_{(0,\infty)} (1-e^{-ux})\,\nu (dx)
\]
where $\lambda\geq 0$ is the drift parameter and $\nu$ is a L\'evy measure, i.e.\ a measure supported in $(0,\infty)$ such that
$\int_{(0,\infty)}\min\{1,x\}\,\nu(dx)<\infty$. For simplicity, we will assume that $\lambda = 0$. In order to exclude the case of compound Poisson processes, we will further assume that $\nu(0,\infty)=\infty$; in particular, the sample paths $t\mapsto \Tp_t$ are a.s.\ \emph{strictly} increasing.

The first-passage time process of the subordinator $\Tp$ is the (generalized, right-continuous) inverse
\begin{align}\label{inv_subordinator}
    \Sp(t)=\inf\{\tau>0\,:\, \Tp(\tau)>t \},\quad t\geq 0.
\end{align}
We will call $\Sp(t)$ an \emph{inverse subordinator}. The assumption $\nu(0,\infty)=\infty$ guarantees that the sample paths of $\Tp$ are a.s.\ strictly increasing, i.e.\ almost all paths $t\mapsto \Sp(t)$ are continuous.
Since the closure of the range of $\Tp$ has zero Lebesgue measure,
the trajectories of $\Sp$ are singular with respect to Lebesgue measure.
Note also that $\Tp$ is a transient L\'evy process and its potential measure satisfies
$U([0,x])=\mathbb{E}(\Sp(x))$. Moreover $U([0,x])/x$ is bounded as $x\rightarrow \infty$, see \cite[p.~41]{Bertoin}.

In what follows we denote by $\hat{h}(u)$ the Laplace transform of a function $h(t)$, i.e.
\[
    \hat{h}(u)=\int_0^\infty e^{-ut}h(t)\,dt.
\]
It is straightforward to verify that
\begin{equation}
\label{Laplace_Inv_Sub}
\hat{s}(x,u)=\frac{f(u)}{u} e^{-x f(u)},
\end{equation}
where $s(x,t)$ is the probability density function of $\Sp(t)$.
In particular, we get for any independent, exponentially distributed random time $\tau\sim u\,e^{-ut}\,dt$
\begin{equation}\label{random-laplace}\begin{aligned}
    \mathbb E e^{\xi \cdot\Sp(\tau)}
    &= \int_0^\infty\kern-5pt\int_0^\infty e^{\xi x} s(x,t) u\,e^{-ut}\,dt\,dx\\
    &= \int_0^\infty e^{\xi x} f(u) \,e^{-x f(u)}\,dx
    = \frac{f(u)}{f(u)-\xi}.
\end{aligned}\end{equation}

Inverse subordinators have found many applications in probability theory.
For their relationship with local times of some Markov processes, see \cite{Bertoin}.
Similarities between inverse subordinators and renewal processes
were investigated in \cite{Bertoin2,Harn,Lageras}.
Applications of inverse subordinators in finance and physics
are discussed in \cite{Winkel} and \cite{Meer1,Jurlewicz_Weron,SWW}, respectively.

In this paper we study some sample path properties of the process
\begin{equation}\label{subdiffusion}
    \Xp(t)
    =B(\Sp(t)), \quad t\geq 0.
\end{equation}
Here, $B$ is a standard one-dimensional Brownian motion (Wiener process)
and $\Sp$ is an inverse subordinator; we assume that $\Sp$ and $B$ are independent. For every jump of the subordinator $\Tp$
there is a corresponding flat period of its inverse $\Sp$. These flat periods represent \emph{trapping events} in which the test particle gets immobilized in a trap. Trapping slows down the overall dynamics of the diffusion process $B$, therefore $\Xp(t)=B(\Sp(t))$
is called a \emph{subdiffusion process}.

In Physics, subdiffusive dynamics is frequently described by fractional diffusion equations \cite{metzler2000}.
Using the relation $\mathbb{P}(\Sp(t)\leq s) = \mathbb{P}(\Tp(s)\geq t)$ and the Fourier-Laplace transform in $x$ and $t$, respectively, one sees that the probability density function $p(x,t)$ of the process $\Xp(t)$
satisfies the following generalized diffusion equation
\begin{equation}\label{ffp_id}
    \frac{\partial p(x,t)}{\partial t}
    =
    \frac{1}{2} {\Phi}_t \frac{\partial^2}{\partial x^2} p(x,t),
\end{equation}
cf.\ \cite{Klaf_Sok,Magdziarz1,Magdziarz3}; this is also a special case of Theorem \ref{thm-2} below.

%

Here, ${\Phi}_t$ is an integro-differential operator given by
\begin{equation}\label{operator}
    \Phi_t g(t)=\frac{d}{dt} \int_0^t M(t-s)g(y)\,ds
\end{equation}
for sufficiently smooth functions $g$. The memory kernel $M(t)$ is defined via its Laplace transform
\begin{equation}\label{kernel}
    \hat{M}(u)
    =\int_0^\infty e^{-ut}M(t)\,dt
    =\frac{1}{f(u)}.
\end{equation}
One can also write \eqref{ffp_id} with an operator $\Phi_t^{-1}$ on the left, i.e.\ a truly ``fractional'' time evolution. This operator is also a
convolution operator, where the convolution kernel has the Laplace transform $f(u)/u$.

The most important case in applications is the case of $\alpha$-stable waiting times,
which corresponds to $f(u)= u^\alpha$ with $\alpha\in(0,1)$ being the stability parameter.
In this case, the operator $\Phi_t$ is the Riemann-Liouville fractional derivative $_0D_t^{1-\alpha}$, cf.\  \cite{samko},
\[
    \Phi_t g(t)
    ={_0D_t}^{1-\alpha} g(t)
    =\frac{1}{\Gamma(\alpha)} \frac{d}{dt}\int_0^t (t-y)^{\alpha-1} g(y)\,dy.
\]
Then, Eq.~\ref{ffp_id} becomes the fractional diffusion equation \cite{MBK,Meer4,Meer1,Meer2,woyczynski}.
The case of space-time fractional derivative operators can be found in \cite{Baeumer1a}.
Subdiffusion phenomena with $\alpha$-stable waiting times have been empirically confirmed
in a number of settings: Charge carrier transport in amorphous semiconductors, nuclear magnetic
resonance, diffusion in percolative and porous systems, transport on fractal geometries, dynamics
of a bead in a polymeric network, and protein conformational dynamics, see \cite{metzler2000}.
The corresponding process $\Xp$ appears in a natural way as a scaling limit
of uncoupled continuous-time random walk with heavy-tailed waiting times \cite{Meer5,Meer6}.
Path properties of $\Xp$ with $f(u)= u^\alpha$ have been investigated in \cite{Meer3,Nane,Magdziarz2}.

Another important example is the tempered-stable case \cite{bib:Cartea2007,Rosinski}, which corresponds to
$f(u)=(u+\lambda)^\alpha-\lambda^\alpha$ with $\lambda>0$ and $0<\alpha<1$.
This type of anomalous dynamics has recently been investigated in \cite{Magdziarz3,Meer7,Meer2a}.
Tempered stable distributions are particularly attractive in the modeling
of transition from the initial subdiffusive character of motion to the standard diffusion for long times.
In physics, applications of tempered stable distributions
in the context of astrophysics and relaxation can be found in \cite{SWW}.
Modeling of the lipid granules dynamics with the use of tempered subdiffusion
has recently been proposed in \cite{Lene}.
Note that the subordinate process $B(\Tp(t))$, where $\Tp(t)$ is the tempered stable
subordinator, is the relativistic stable process
with Fourier exponent equal to $(\xi^2 + m^{2/\beta})^{\beta/2} - m$, \cite{relativistic}.
Here $\lambda = m^{2/\beta}$ and $\alpha=\beta/2$.

Another interesting example are distributed order fractional diffusion equations \cite{Kochubei,Chechkin,Meer9},
which correspond to the case $f(u)=\int_0^\infty (1-e^{-ux})\,\nu (dx)$ with
$\nu(t,\infty)=\int_0^1 t^{-\beta} \,\mu(d\beta)$. Here, $\beta\in (0,1)$ and $\mu$ is
some distribution supported in $[0,1]$.
This type of Laplace exponent leads to the ultra-slow dynamics displayed by $\Xp$.
The corresponding fractional Cauchy problem has recently been analyzed in \cite{Mijena}.

Our paper is organized as follows: In the next section, we will investigate various properties of
the subdiffusion process $\Xp$. In particular, we will describe its martingale properties,
derive a generalized Feynman-Kac formula, and the laws of large numbers of the iterated logarithm for $\Xp$.
Our results extend those for the $\alpha$-stable case in \cite{Nane,Magdziarz2}.
In Section 3 we will construct a stationary sequence of increments
of the appropriately modified process $\Xp$. We will show that this sequence is
ergodic and mixing. This is the main result of the paper; it can be applied to
verify ergodicity and mixing of all subdiffusive complex systems modeled by the generalized diffusion equation \eqref{ffp_id}.

\section{Asymptotic behaviour of trajectories}

We begin with a result which establishes martingale properties of the subdiffusion process $\Xp$.
\begin{theorem}\label{thm-1}
\textup{(i)}
    Both $\Xp(t)$ and $\Yp(t):=\exp\left\{ \Xp(t)-\frac{1}{2}\Sp (t) \right\}$, $t\geq0$, are martingales with respect to the probability measure $\mathbb{P}$.

\smallskip
\textup{(ii)}
    Let $t\in [0,T]$. Then for each $\epsilon\geq 0$ the process $\Zp(t)=\exp\left\{ \Xp(t)\right\}$ is a  martingale with respect to the probability measure $\mathbb{P}_\epsilon$ defined as
    \begin{equation*} 
        \mathbb{P}_\epsilon(A)
        =c_\epsilon \cdot\int_A  \exp\left\{ -\frac{1}{2} \Xp(T)-\left( \epsilon+\frac{1}{8}\right)\Sp(T) \right\} d\mathbb{P}.
    \end{equation*}
    Here $A\in\mathcal{F}$ and $c_\epsilon^{-1}= \int_\Omega  \exp\left\{ -\frac{1}{2} \Xp(T)-\left( \epsilon+\frac{1}{8}\right)\Sp(T) \right\} d\mathbb{P}$ is the normalizing constant.
\end{theorem}

\begin{proof}
(i) The proof of this part is similar to the proof \cite[Theorem 2.1]{Magdziarz2}. Note first that the quadratic variation of $\Xp(t)$ satisfies
$
    \langle \Xp,\Xp\rangle_t = \vphantom{\rangle}^{f}\!\langle B,B\rangle_t  = \Sp(t).
$
Set
\begin{equation}\label{filtr}
    \mathcal{F}_t
    =\bigcap_{u>t}  \sigma \big(\{B(y)\,:\, 0\leq y\leq u\},\; \{\Sp(t)\,:\, t\geq 0\}\big).
\end{equation}
By definition, $\{\mathcal{F}_t\}$ is right-continuous, $(B(t),\mathcal F_t)_{t\geq 0}$ is a martingale,
and for every fixed $t_0>0$ the random variable $\Sp(t_0)$ is a stopping time with respect to $\{\mathcal{F}_t\}$.
Thus $\{\mathcal G_t\}$ where $\mathcal{G}_t=\mathcal{F}_{\Sp(t)}$ is a well defined filtration.
Let us introduce the sequence of $\{\mathcal{F}_t\}$-stopping times
\[
    T_n = \inf\{u>0\,:\, |B(u)|=n\}.
\]
By Doob's optional sampling theorem we have for $s<t$
\[
    \mathbb{E}(B(T_n\wedge \Sp(t)) \;|\; \mathcal{G}_s)= B(T_n\wedge \Sp(s)).
\]
The right-hand side of the above equation converges to $B(\Sp(s))$, whereas the left-hand side converges to $E(B(\Sp(t)) \;|\; \mathcal{G}_s)$ as $n\rightarrow \infty$. Therefore, $\Xp$ is a martingale.

From \cite[Proposition IV.3.4]{RevuzYor} we get that $\Yp$ is a local martingale. To prove that it is also a martingale, we have to
verify the integrability of the random variable $\sup_{0\leq u \leq t} \Yp(u)$.
Using Doob's maximal inequality and the independence of $B$ and $\Sp$, we get
\[
    \mathbb{E}\left[\left ( \sup_{0\leq u \leq t} \exp\left \{ B(\Sp(u))\right \}\right)^2\right]
    \leq
    4 \mathbb{E}\left[\exp\left \{2 B(\Sp(t))\right \}\right]
    =
    4 \mathbb{E}\left( \exp\left \{2 \cdot\Sp(t)\right \}\right);
\]
Assume that the parameter $u$ of $\tau$  in \eqref{random-laplace} is so large that $f(u)>2$. Then we see that the expression on the right-hand side is finite. Thus,
$\mathbb{E}\big[\sup_{0\leq u \leq t} \Yp(u)\big]<\infty$, and $\Yp$ is a martingale.

\bigskip\noindent
(ii) Put
\[
    V(t)=\exp\left\{ -\frac{1}{2}B(t)-\frac{1}{8} t  \right\}
    \quad\text{and}\quad
    W(t)=\exp\left\{ B(t)\right\}.
\]
Then
\[
    V(t)W(t)=\exp\left\{ \frac{1}{2} B(t)-\frac{1}{8} t  \right\}.
\]
Therefore, the product $V\cdot W$ is an $(\mathcal{F}_t,\mathbb{P})$-martingale, and so is the stopped process $V(\cdot\wedge \Sp(T))W(\cdot\wedge \Sp(T))$.
Since the bounded random variable $e^{-\epsilon \Sp(T)}$ is $\mathcal{F}_0$-measurable, it follows that $e^{-\epsilon \Sp(T)}V(\cdot\wedge \Sp(T))W(\cdot\wedge \Sp(T))$ is an $(\mathcal{F}_t,\mathbb{P})$-martingale. This in turn implies that $W(\cdot\wedge \Sp(T))$
is an $(\mathcal{F}_t,\mathbb{P}_\epsilon)$-martingale.

Denote by $\mathbb{E}^{\mathbb{P}_\epsilon}$ the expectation with respect to the measure $\mathbb{P}_\epsilon$. Then, one shows similarly to (i) that
\begin{equation*}
    \mathbb{E}^{\mathbb{P}_\epsilon} \left(  \sup_{t\geq 0}W(t\wedge \Sp(T))\right)
    = \mathbb{E}^{\mathbb{P}_\epsilon} \left(  \sup_{t\leq \Sp(T)}W(t)\right)
    <\infty.
\end{equation*}
Thus, $W(\cdot\wedge \Sp(T))$ is a uniformly integrable $(\mathcal{F}_t,\mathbb{P}_\epsilon)$-martingale.
Consequently, there exists a random variable $H$ such that $W(t\wedge \Sp(T))=\mathbb{E}^{\mathbb{P}_\epsilon} \left( H|\mathcal{F}_t\right)$. Moreover
\[
    \Zp(t)=W(\Sp(t)\wedge \Sp(T))=\mathbb{E}^{\mathbb{P}_\epsilon} \left( H|\mathcal{F}_{\Sp(t)}\right),
\]
and this implies that $\Zp$ is an $(\mathcal{F}_{\Sp(t)},\mathbb{P}_\epsilon)$-martingale.
\end{proof}

The next result is a generalization of the Feynman-Kac formula.
Consider the standard diffusion process $V$ given by the following It\^{o} stochastic
differential equation
\begin{equation*}\begin{aligned}
    dV(t)   &=  \mu(V(t))\,dt+\sigma(V(t))\,dB(t)\\
    V(0)    &=  x.
\end{aligned}\end{equation*}
In order to ensure the existence and uniqueness of the solution we assume the usual (local)
Lipschitz and (global) growth conditions on $\mu$ and $\sigma$. Then the infinitesimal generator of $V(t)$ is
\[
    \mathcal{A}\phi(x)
    =\mu(x)\frac{\partial}{\partial x}\,\phi(x)+\frac{1}{2}\sigma^2(x)\frac{\partial^2}{\partial x^2}\phi(x),
    \quad\phi\in C_c^2(\mathbb R).
\]
\begin{theorem}\label{thm-2}
Let $g$ and $h$ be continuous functions, such that $g$ is bounded and $h\geq 0$.
Then the function
\begin{equation}\label{FK_function}
    v(x,t)=\mathbb{E}^x \left( \exp\left\{ - \int_0^{\Sp (t)}h(V(u))\,du  \right\} g(V(\Sp (t)))\right)
\end{equation}
satisfies the generalized Feynman-Kac equation
\begin{equation}\label{FK_eq}
    \frac{\partial v(x,t)}{\partial t}=\Phi_t\left[\mathcal{A}v(x,t)-h(x)v(x,t)\right]
\end{equation}
with the initial condition $v(x,0)=g(x)$ and the operator $\Phi_t$ from \eqref{operator}.
\end{theorem}
\begin{proof}
Denote by $w(x,t)$ the function
\[
    w(x,t)=\mathbb{E}^x \left( \exp\left\{ - \int_0^{t}h(V(u))\,du  \right\} g(V(t))\right)
\]
and recall that $s(y,t)$ is the probability density of $\Sp(t)$. With the above assumptions on $g$ and $h$, $w(x,t)$ is in $C^{2,1}(\mathbb{R}\times\mathbb{R_+})$.

Using the classical Feynman-Kac formula, we get that $w(x,t)$
satisfies in the Laplace space the equation
\begin{equation}\label{Laplace_w}
    u\hat{w}(x,u)-w(x,0)
    =\mathcal{A}\hat{w}(x,u)-h(x)\hat{w}(x,u).
\end{equation}
Since $\Sp$ and $B$ are independent, we obtain
\begin{equation*}
    v(x,t)
    =\int_\mathbb{R} w(x,y) s(y,t)\,dy,
\end{equation*}
which by \eqref{Laplace_Inv_Sub} implies that
\begin{equation}\label{main_Laplace}
    \hat{v}(x,u)=\frac{f(u)}{u}\hat{w}(x,f(u)).
\end{equation}
Replacing $u \rightarrow f(u)$ in \eqref{Laplace_w} we get
\begin{equation*}
    f(u)\hat{w}(x,f(u))-w(x,0)
    =\mathcal{A}\hat{w}(x,f(u))-h(x)\hat{w}(x,f(u)).
\end{equation*}
Consequently, applying \eqref{main_Laplace} and the fact that $w(x,0)=v(x,0)$  we conclude that $\hat{v}(x,u)$ satisfies
\begin{equation*}
    u\hat{v}(x,u)-v(x,0)
    =\frac{u}{f(u)}\left[\mathcal{A}\hat{v}(x,u)-h(x)\hat{v}(x,u)\right].
\end{equation*}
Inverting the Laplace transform, we get that $v(x,t)$ satisfies \eqref{FK_eq}.
\end{proof}

We will now study the asymptotic behaviour of the trajectories of $\Xp$.
\begin{theorem}[Law of large numbers]
The trajectories of $\Xp(t)$ satisfy
\[
    \lim_{t\rightarrow\infty}\frac{\Xp(t)}{t}=0 \;\;\;a.s.
\]
\end{theorem}
\begin{proof}
Fix $\epsilon>0$ and define
\[
    A_n=\left \{\sup_{2^n\leq t \leq 2^{n+1}}\left |\frac{\Xp(t)}{t}\right|>\epsilon \right\}.
\]
Using Markov's inequality we get that
\[
    \epsilon^2 \mathbb{P}(A_n)
    \leq \mathbb{E} \left[\left(\sup_{2^n\leq t \leq 2^{n+1}}\left |\frac{\Xp(t)}{t}\right|\right)^2\right]
    \leq
    \frac{\mathbb{E} \left[\left(\sup_{2^n\leq t \leq 2^{n+1}}\left|\Xp(t)\right|\right)^2\right]}{2^{2n}}.
\]
Using the fact that  $\Xp(t)$ is a martingale, we obtain from the Doob's maximal inequality the following bound
\[
    \mathbb{E} \left[\left(\sup_{2^n\leq t \leq 2^{n+1}}\left|\Xp(t)\right|\right)^2 \right]
    \leq
    4\mathbb{E} \left(\big|\Xp(2^{n+1})\big|^2\right)=4\mathbb{E} \left(\Sp(2^{n+1})\right).
\]
Since $\mathbb{E}(\Sp(x))/x=U([0,x])/x$ is bounded as $x\rightarrow \infty$,
thus
\[
    \mathbb{E} \left(\Sp(2^{n+1})\right)\leq c\,2^{n+1}
\]
for large $n$ and some positive constant $c$. Thus
\[
    \mathbb{P}(A_n)\leq \frac{4c\,2^{n+1}}{2^{2n}\epsilon^2}
\]
for large enough $n$, and so
\[
    \sum_{n=1}^\infty \mathbb{P}(A_n)<\infty.
\]
Finally, using the Borel-Cantelli lemma we obtain the desired result.
\end{proof}

Now, we turn to the law of the iterated logarithm. It is known that for the inverse subordinator $\Sp(t)$ we have, see \cite{Fristedt},
\[
    \limsup_{t\rightarrow \infty}\frac{\Sp(t)}{\phi(t)}=1
\]
for a certain function $\phi(t)$.
Thus, by the celebrated law of the iterated logarithm for a Brownian motion $B$, we obtain
the following bound for $\Xp(t)=B(\Sp(t))$
\[
    \limsup_{t\rightarrow \infty} \frac{\Xp(t)}{\sqrt{2\phi(t)\log\log (\phi(t))}}\leq 1.
\]
This result is, however, not sharp, cf.\ \cite{Magdziarz2}: Just observe that the large increments of $B$ do not
coincide with those of $\Sp$. The next theorem determines precisely the asymptotic behaviour of the trajectories of
$\Xp$ for small and large times.

\begin{theorem}[Law of the iterated logarithm]
{\upshape (i)}
    Assume that for some $\epsilon>0$ the Laplace exponent of $\Tp$ satisfies $f(u)\geq u^\epsilon$ for sufficiently large $u$. Then, for $\gamma>1$ there exists a constant $c>0$ such that
    \[
        \limsup_{t\searrow 0}\frac{\Xp(t)}{g(t)}=1
        \quad\text{and}\quad
        \liminf_{t\searrow 0}\frac{\Xp(t)}{g(t)}=-1
        \quad\text{a.s.}
    \]

{\upshape (ii)}
    Assume that for some $\epsilon>0$ the Laplace exponent of $\Tp$ satisfies $f(u)\leq u^\epsilon$ for sufficiently small $u$. Then, for $\gamma>1$ there exist a constant $c>0$ such that
    \[
        \limsup_{t\nearrow \infty}\frac{\Xp(t)}{g(t)}=1
        \quad\text{and}\quad
        \liminf_{t\nearrow \infty}\frac{\Xp(t)}{g(t)}=-1
        \quad\text{a.s.}
    \]

    Here, the function $g(t)$ is the inverse of
    \begin{equation}\label{normalizing_fun}
        h(t)
        =\frac{c\log(|\log (t)|)}{\eta\left(\gamma t^{-1}\log(|\log (t)|)\right)},
    \end{equation}
    and $\eta(u)$ is the inverse of $\varrho(u)=\sqrt{2f(u)}$.
\end{theorem}
\begin{proof}
In the proof we will apply the method used in \cite{Bertoin3} in the context of iterated Brownian motion.

\bigskip\noindent
(i) Put
\[
    B^*(t)=\sup_{s\in [0,t]}B(s)
\]
and
\[
    \Xp^*(t)=\sup_{s\in [0,t]}\Xp(s)=B^*(\Sp(t)).
\]
Let us introduce
\[
S_{1/2}(t) {=}\inf\{\tau:B^*(\tau)>t\}
\]
It is known \cite{Sato} that $S_{1/2}(\tau)$ is the $1/2$-stable subordinator with Laplace transform
\[
    \mathbb{E}\left( e^{-u S_{1/2}(\tau)} \right ) = e^{-\tau \sqrt{2u}}.
\]
Additionally, we have
\[
    \Xp^*(t) {=} \inf\big\{\tau\,:\, \Tp(S_{1/2}(\tau))>t\big\}.
\]
The Laplace transform of the subordinator $\Tp(S_{1/2}(t))$ is given by
\[
    \mathbb{E}\left( e^{-u \Tp(S_{1/2}(t))} \right)
    =
    \mathbb{E}\left( e^{- S_{1/2}(t) f(u)} \right )=e^{-t\sqrt{2f(u)}}.
\]
Using the assumptions on the L\'evy exponent $f(u)$, we can apply \cite[Theorem 1]{Fristedt} to get
\[
    \liminf_{t\searrow 0}\frac{\Tp(S_{1/2}(t))}{h(t)}=1,
\]
where $h(t)$ is defined in \eqref{normalizing_fun}.

Since
\[
    B^*(\Sp(\Tp(S_{1/2}(t))))=t,
\]
we obtain that
$$
    \mathbb P\big(B^*(\Sp(c_0 h(t)))\geq t\text{\ \ infinitely often as\ \ } t\searrow 0\big)
    =
    \begin{cases}
    1,  &\text{if\ } c_0>1,\\
    0,  &\text{if\ } c_0<1.
    \end{cases}
$$
This in turn implies that
\[
    \limsup_{t\searrow 0}\frac{B^*(\Sp(t))}{g(t)}=1,
\]
where $g(t)$ is the inverse of $h(t)$. Clearly, $B^*$ in the numerator can be replaced by $B$, and we obtain
\[
    \limsup_{t\searrow 0}\frac{\Xp(t)}{g(t)}=1,
\]
which proves the first half of part (i).

The result
\[
    \liminf_{t\searrow 0}\frac{\Xp(t)}{g(t)}=-1,
\]
is a consequence of the symmetry of $B$.

\bigskip
The proof of part (ii) of the theorem is analogous.
\end{proof}

\begin{rem}
    In some particular cases it is possible to find the explicit form of the normalizing function $g$.
    For example in the $\alpha$-stable case $g(t)\sim c_0 t^{\alpha/2}(\log|\log t |)^{(2-\alpha)/2}$
as $t\rightarrow 0^+$ and $t\rightarrow \infty$, where $c_0$ is an appropriate positive constant, cf.~\cite{Nane,Magdziarz2}.
\end{rem}

\section{Ergodic properties}
Our standard reference for the ergodic theory of dynamical systems is \cite{Lasota}. For the convenience of our readers let us recall some basic facts.

Let  $(X, \mathcal{A}, \upsilon, S)$ be a measure-preserving dynamical system:
$X$ is the phase space, $\mathcal{A}$ is a $\sigma$-algebra on $X$, $\upsilon$ is a probability
measure on $X$, and  $S: X \rightarrow X$ is a measure-preserving transformation. A set $A \in\mathcal{A}$ is invariant if $S^{-1}(A) = A$.
We say that $(X, \mathcal{A}, \upsilon, S)$ is \emph{ergodic} if every invariant set $A\in\mathcal{A}$ is trivial, i.e.\ if either $\upsilon(A) = 0$ or $\upsilon(X\setminus A) = 0$.
We say that $(X, \mathcal{A}, \upsilon, S)$ is \emph{mixing}, if for all $A,B\in\mathcal{A}$
\[
    \lim_{n\rightarrow\infty}\upsilon(A\cap S^{-n}B)=\upsilon (A)\upsilon(B).
\]
Clearly, mixing is stronger than ergodicity.

Ergodicity and mixing have their origins in statistical physics.
Intuitively, a system is ergodic if the phase space $X$ cannot be divided into two regions such that a phase point
starting in one region will always stay in that region. Thus, for ergodic transformations
every phase point will visit the whole phase space. Another important property of ergodic systems
is that their temporal and ensemble averages coincide, which follows from Birkhoff's ergodic theorem, cf.\ \cite{Lasota}.
Mixing can be viewed as the asymptotic independence of the sets $A$ and $B$ under the transformation $S$.

In the context of stochastic processes, we use the following setup.
Consider a real-valued stationary process $Y(n)$, $n\in\mathbb{N}$. In its canonical representation,
$\{Y(n)\}$ can be identified with its law which is a probability measure $\mathbb{P}$ on
the space $\mathbb{R}^\mathbb{N}$. On this space we use
the canonical $\sigma$-algebra $\mathcal{B}$ generated by the cylinder sets,
and we consider the standard shift transformation $S: \mathbb{R}^\mathbb{N} \rightarrow \mathbb{R}^\mathbb{N}$.
Stationarity of $\{Y(n)\}$ implies that the shift $S$ is measure-preserving.
Note that $(\mathbb{R}^\mathbb{N},\mathcal{B},\mathbb{P},S)$ is a typical object of study
in the theory of dynamical systems. Therefore, ergodic properties of stationary stochastic processes
can be studied in the framework of the theory of dynamical systems.
Detailed analysis of ergodicity and mixing for the classes of infinitely divisible and fractional processes,
can be found in \cite{Janicki_Weron,Cambanis,Cambanis2,Rosinski_Zak} and \cite{Magdziarz4}, respectively.

Here we will study ergodic properties of the anomalous diffusion process $\Xp$.
However, one cannot verify ergodicity and mixing of $\Xp$ in a straightforward manner,
since neither $\Xp$ nor its increments are stationary.
Therefore, we will introduce a modification.

Let us assume that the first moment of $\Tp$ is finite: $\mu:= \mathbb{E}(\Tp(1))<\infty$.
Let
\[
    \tilde T^f(t)=T_0 + \Tp(t)
\]
denote the generalized subordinator whose initial distribution is given by
\[
    \mathbb{P}(T_0\leq x)=\frac{1}{\mu} \int_0^x \int_y^\infty \nu(ds)\,dy
\]
($T_0$ is assumed independent of $\Tp$). The corresponding inverse subordinator is
\begin{equation}\label{stationary_subordinator}
    \tSp(t)=\inf\big\{\tau>0\,:\, \tilde T^f(\tau)>t \big\}.
\end{equation}
With this choice of the initial distribution of $T_0$,
the process $\tSp$ has stationary increments, cf.\ \cite{Lageras}.
Consequently, the modification of $\Xp$ defined as
\[
    \tXp (t)=B(\tSp(t)),
\]
has stationary increments as well.
We denote this stationary sequence by
\begin{equation}\label{stationary_increments}
    \tYp(n)=\tXp(n+1)-\tXp(n),\quad n\in\mathbb{N}.
\end{equation}
Now we are in position to verify its ergodicity and mixing
\begin{theorem}
\label{Th_mixing}
The stationary sequence $\tYp$ defined in \eqref{stationary_increments} is ergodic and mixing.
\end{theorem}
Before we embark on the proof of the above theorem, we need two lemmas.
The first shows that the increments of $\Sp$ are asymptotically stationary.
More precisely, the increments of $\Sp$ converge in distribution to the increments of $\tSp$.
\begin{lemma}\label{lemma1}
Let $M\in\mathbb{N}$, $z_j\in\mathbb{R}_+$, $b_j\in \mathbb{R}_+$, $j=1,\ldots,M$, with $z_i<z_j$ for $i<j$. Then
\begin{align*}
    \mathbb{E}\bigg( \exp &\bigg \{ -\sum_{j=1}^M b_j \big(\Sp(z_j+1+n)-\Sp(z_j+n)  \big) \bigg \}  \bigg) \\
    &\xrightarrow[n\to\infty]{} \mathbb{E}\bigg( \exp \bigg \{ -\sum_{j=1}^M b_j \big(\tSp(z_j+1)-\tSp(z_j)  \big) \bigg \}  \bigg)
\end{align*}
\end{lemma}
\begin{proof}
Assume that $M=1$. We will show that
\begin{equation}\label{conv_inc}\begin{aligned}
    \mathbb{E}\bigg( \exp &\bigg \{ - b_1 \big(\Sp(z_1+1+n)-\Sp(z_1+n)  \big) \bigg \}  \bigg)\\
    &\xrightarrow[n\to\infty]{} \mathbb{E}\bigg( \exp \bigg \{ - b_1 \big(\tSp(z_1+1)-\tSp(z_1)  \big) \bigg \}  \bigg).
\end{aligned}\end{equation}Observe that $\tSp(t)\leq \Sp(t)$ and that, by \eqref{random-laplace} and the continuity of $t\mapsto \Sp(t)$, $\Sp(t)$ has exponential moments for any $t>0$. Therefore, the moments determine the distribution
of $\tSp(t)$, cf.\ \cite[Chapter VII, Section 3]{Feller}.
Consequently, the distribution of the random variable $\tSp(z_1+1)-\tSp(z_1)$
is also determined by its moments. This means that it is enough to prove the convergence of moments
\begin{equation}\label{conv_moments}
    \mathbb{E}\left[\big(\Sp(z_1+1+n)-\Sp(z_1+n)\big)^k\right]
    \xrightarrow[n\to\infty]{} \mathbb{E}\left[\big(\tSp(z_1+1)-\tSp(z_1)\big)^k\right]
\end{equation}
for every $k\in\mathbb{N}$.

Denote by $U(t)=\mathbb{E}(\Sp(t))$ the renewal function of $\Sp$ and by $U(dx)$ the measure induced by $U$. We introduce another measure on $[z_1,z_1+1]$ defined as
\[
    U_n[z_1,x]=U(x+n)-U(n),\;\;\; x\in [z_1,z_1+1].
\]
Applying the renewal theorem we get that

\[
    U_n[z_1,x]\xrightarrow[n\to\infty]{} \frac{x}{\mu} \;\;\;\;\;\; \text{pointwise},
\]
where $\mu=\mathbb{E}(\Tp(1))$. This implies that
\[
    U_n \xrightarrow[n\to\infty]{} \frac{\lambda}{\mu} \;\;\;\;\;\; \text{weakly},
\]
where $\lambda$ is Lebesgue measure on $[z_1,z_1+1]$.

Now, using \cite[Theorem 1]{Lageras}, we obtain
\begin{align*}
    \mathbb{E}\Big[\big(&\Sp(z_1+1+n)-\Sp(z_1+n)\big)^k\Big] \\
    &= k! \int\limits_{z_1+n}^{z_1+1+n}\int\limits_{x_1}^{z_1+1+n} \ldots \int\limits_{x_{k-1}}^{z_1+1+n} U(dx_k-x_{k-1})\ldots U(dx_2-x_{1})U(dx_1) \\
    &= k! \int\limits_{z_1}^{z_1+1}\int\limits_{y_1}^{z_1+1} \ldots \int\limits_{y_{k-1}}^{z_1+1} U(dy_k-y_{k-1})\ldots U(dy_2-y_{1})U_n(dy_1),
\end{align*}
where the last equality was obtained by changing the variables $y_i=x_i-n$, $i=1,\ldots,k$. By the weak convergence of $U_n$ to $\lambda/\mu$ and \cite[Theorem 1]{Lageras}, we get
\begin{align*}
    \mathbb{E}\Big[\big(&\Sp(z_1+1+n)-\Sp(z_1+n)\big)^k\Big]\\
    &\xrightarrow[n\to\infty]{} k! \int\limits_{z_1}^{z_1+1}\int\limits_{y_1}^{z_1+1} \ldots \int\limits_{y_{k-1}}^{z_1+1} U(dy_k-y_{k-1})\ldots U(dy_2-y_{1})\frac{dy_1}{\mu}\\
    &= \mathbb{E}\Big[\big(\tSp(z_1+1)-\tSp(z_1)\big)^k\Big].
\end{align*}
This ends the proof for $M=1$.

The proof for arbitrary $M\in \mathbb{N}$ is similar. Note first that the random variable
\[
    \sum_{j=1}^M b_j \big(\tSp(z_j+1)-\tSp(z_j)  \big)
\]
is determined by its moments. So, it is enough to show
\begin{equation*}
    \mathbb{E}\Bigg[\bigg( \sum_{j=1}^M b_j \big(\Sp(z_j+1+n)-\Sp(z_j+n)  \big) \bigg)^k\Bigg] \xrightarrow[n\to\infty]{}
    \mathbb{E}\Bigg[\bigg( \sum_{j=1}^M b_j \big(\tSp(z_j+1)-\tSp(z_j)  \big) \bigg)^k\Bigg]
\end{equation*}
for any $k\in\mathbb{N}$. As before, using the multinomial formula
\[
    \bigg( \sum_{j=1}^M  x_j \bigg)^k  = \sum_{\begin{subarray}{c} k_1,...,k_M \geq 0\\ k_1+\cdots+k_M=k\end{subarray}}
    \frac{k!}{k_1 !\ldots k_M !} x_1 ^{k_1}\ldots x_M^{k_M},
\]
one needs to show the convergence
\begin{equation}\label{conv_moments3}
    \mathbb{E}\Bigg[ \prod_{j=1}^M \big(\Sp(z_j+1+n)-\Sp(z_j+n)  \big)^{k_j} \Bigg]
    \xrightarrow[n\to\infty]{} \mathbb{E}\Bigg[ \prod_{j=1}^M  \big(\tSp(z_j+1)-\tSp(z_j)  \big)^{k_j} \Bigg]
\end{equation}
with $\sum{k_j}=k$. Applying \cite[Theorem 1]{Lageras}, we obtain
\begin{equation*}
    \mathbb{E}\Bigg( \prod_{j=1}^M \big(\Sp(z_j+1+n)-\Sp(z_j+n)  \big)^{k_j} \Bigg)
    =
    \prod_{i=1}^M k_i ! \prod_{j=1}^k \int_C U(dx_j-x_{j-1}),
\end{equation*}
where $C=\{(x_0,\ldots,x_k) \,:\, x_0=0, z_j+n < x_{k_0+\ldots+k_{j-1}+1}<\ldots<x_{k_0+\ldots+k_{j}}\leq z_j+1+n, j=1,\ldots,M,k_0=0  \}$.
Finally, changing the variables $y_i=x_i-n$, $i=1,\ldots,k$ and
using the weak convergence of $U_n$ to $\lambda/\mu$, we get \eqref{conv_moments3}.
\end{proof}

We will use the following recursive relations for inverse subordinators from \cite{Kaj}:
\begin{align}\label{recurrence}
    &\mathbb{E}\bigg( \exp\bigg\{ \sum_{j=1}^M \theta_j \Sp(\tau_j) \bigg\} \bigg)
    =\mathbb{E}\bigg( \exp\bigg\{ \sum_{j=2}^M \theta_j \Sp(\tau_j) \bigg\} \bigg) + \\ \notag
    &\quad\mbox{}+ \frac{\theta_1}{\sum_{j=1}^M\theta_j}\int_0^{\tau_1} \mathbb{E}\bigg( \exp\bigg\{ \sum_{j=2}^M \theta_j \Sp(\tau_j-x) \bigg\} \bigg) d_x\mathbb{E}\bigg( \exp \bigg\{ \Sp(x) \sum_{j=1}^M\theta_j  \bigg\} \bigg),\\
\label{recurrence2}
    &\mathbb{E}\bigg( \exp\bigg\{ \sum_{j=1}^M \theta_j \tSp(\tau_j) \bigg\} \bigg)
    = \mathbb{E}\bigg( \exp\bigg\{ \sum_{j=2}^M \theta_j \tSp(\tau_j) \bigg\} \bigg)+ \\ \nonumber
    &\quad\mbox{}+ \frac{\theta_1}{\sum_{j=1}^M\theta_j}\int_0^{\tau_1} \mathbb{E}\bigg( \exp\bigg\{ \sum_{j=2}^M \theta_j \Sp(\tau_j-x) \bigg\} \bigg) d_x\mathbb{E}\bigg( \exp \bigg\{ \tSp(x) \sum_{j=1}^M\theta_j  \bigg\} \bigg),
\end{align}
where $M\geq 2$, $0\leq\tau_1\leq...\leq \tau_M$,
$\theta_1,...,\theta_M\in\mathbb{R}$.

We will also need the following technical result
\begin{lemma}
\label{lemma2}
Let $m,M\in\mathbb{N}$ with $1\leq m \leq M$, $0\leq z_1\leq \ldots \leq z_M$, $b_j\in \mathbb{R}$ for $j=1,\ldots,m$ and
$b_j\in \mathbb{R}_+$ for $j=m+1,\ldots,M$. Then
\begin{align}\label{lemma2_formula}
    &\mathbb{E}\bigg(\exp \bigg \{ \sum_{j=1}^m b_j \Sp(z_j) -\sum_{j=m+1}^M b_j \big(\Sp(z_j+1+n)-\Sp(z_j+n)  \big) \bigg \}  \bigg) \\ &\notag\xrightarrow[n\to\infty]{}\mathbb{E}\bigg( \exp \bigg \{ \sum_{j=1}^m b_j \Sp(z_j)  \bigg \}  \bigg) \mathbb{E}\bigg( \exp \bigg \{ -\sum_{j=m+1}^M b_j \big(\tSp(z_j+1)-\tSp(z_j)  \big) \bigg \}  \bigg).
\end{align}
\end{lemma}
\begin{proof}
We will prove formula \eqref{lemma2_formula} by induction on the parameter $M$.

\noindent
Step I. Let $M=1$. Then $m=1$ and the formula is trivially fulfilled.

\noindent
Step II. Assume that the formula \eqref{lemma2_formula} holds for $M-1$. We will
show that it is also true for $M$. Applying \eqref{recurrence} we have
\begin{align*}
    &\mathbb{E}\bigg( \exp \bigg \{ \sum_{j=1}^m b_j \Sp(z_j) -\sum_{j=m+1}^M b_j \big(\Sp(z_j+1+n)-\Sp(z_j+n)  \big) \bigg \}  \bigg)\\
    &=\mathbb{E}\bigg( \exp \bigg \{ \sum_{j=2}^m b_j \Sp(z_j) -\sum_{j=m+1}^M b_j \big(\Sp(z_j+1+n)-\Sp(z_j+n)  \big) \bigg \}  \bigg)\\
    &\quad\mbox{}+ \frac{b_1}{\sum_{j=1}^m b_j} \int_0^{z_1} \mathbb{E}\bigg( \exp\bigg\{ \sum_{j=2}^m b_j \Sp(z_j-x)-\mbox{}\\
    &\quad\mbox{}-\sum_{j=m+1}^M b_j \big(\Sp(z_j+1+n-x)-\Sp(z_j+n-x) \big) \bigg\} \bigg) d_x\mathbb{E}\bigg( \exp \bigg\{ \Sp(x) \sum_{j=1}^m b_j  \bigg\} \bigg).
\end{align*}
Now, if $m>1$ we apply the induction assumption, the dominated convergence theorem to
get that the expression above converges to
\begin{gather*}
    \mathbb{E}\bigg( \exp \bigg \{ \sum_{j=2}^m b_j \Sp(z_j)\bigg \}  \bigg) \mathbb{E}\bigg( \exp \bigg \{  -\sum_{j=m+1}^M b_j \big(\tSp(z_j+1)-\tSp(z_j)  \big) \bigg \}  \bigg) + \\
    \mbox{}+ \mathbb{E}\bigg( \exp \bigg \{  -\sum_{j=m+1}^M b_j \big(\tSp(z_j+1)-\tSp(z_j)  \big) \bigg \}  \bigg)\times \\
    \mbox{}\times\frac{b_1}{\sum_{j=1}^m b_j} \int_0^{z_1} \mathbb{E}\bigg( \exp\bigg\{ \sum_{j=2}^m b_j \Sp(z_j-x) \bigg\} \bigg) d_x\mathbb{E}\bigg( \exp \bigg\{ \Sp(x) \sum_{j=1}^m b_j  \bigg\} \bigg)
\end{gather*}
and this is, by \eqref{recurrence}, equal to
$$
    \mathbb{E}\bigg( \exp \bigg \{ \sum_{j=1}^m b_j \Sp(z_j)  \bigg \}  \bigg) \mathbb{E}\bigg( \exp \bigg \{ -\sum_{j=m+1}^M b_j \big(\tSp(z_j+1)-\tSp(z_j)  \big) \bigg \}  \bigg).
$$
For $m=1$ we apply Lemma \ref{lemma1} and use the same argumentation. This ends the proof.
\end{proof}

Now we are ready to prove our main result.
\begin{proof}[Proof of Theorem \ref{Th_mixing}.]
In order to prove that $\tYp(n)$ defined in \eqref{stationary_increments} is mixing it is sufficient to
show the following convergence of characteristic functions, cf.\ \cite{Maruyama},
\begin{equation*}\begin{aligned}
    \mathbb{E}\bigg( \exp &\bigg( i\sum_{j=1}^m a_j \tYp({z_j})+ i\sum_{j=m+1}^M a_j \tYp({z_j+n}) \bigg)  \bigg ) \\ &\xrightarrow[n\to\infty]{} \mathbb{E}\bigg( \exp \bigg( i\sum_{j=1}^m a_j \tYp({z_j}) \bigg)  \bigg) \mathbb{E}\bigg( \exp \bigg( i\sum_{k=m+1}^M a_j \tYp({z_j}) \bigg)  \bigg),
\end{aligned}\end{equation*}
for $m,M\in\mathbb{N}$, $1\leq m \leq M$, $a_1,...,a_M \in \mathbb{R}$ and $0\leq z_1\leq...\leq z_M$. By the
definition of $\tYp(t)$ we get that the above
convergence is equivalent to
\begin{equation*}\begin{aligned}
    \mathbb{E}\bigg(\exp &\bigg \{ -\sum_{j=1}^m c_j \big( \tSp(z_j+1)-\tSp(z_j)\big) -\sum_{j=m+1}^M c_j \big(\tSp(z_j+1+n)-\tSp(z_j+n)\big) \bigg \} \bigg) \\
    &\xrightarrow[n\to\infty]{} \mathbb{E}\bigg( \exp \bigg \{ -\sum_{j=1}^m c_j\big( \tSp(z_j+1)-\tSp(z_j)\big)  \bigg \}  \bigg)\times\mbox{}\\ &\phantom{\xrightarrow[n\to\infty]{}}\quad\mbox{}\times \mathbb{E}\bigg( \exp \bigg \{ -\sum_{j=m+1}^M c_j \big(\tSp(z_j+1)-\tSp(z_j)  \big) \bigg \}  \bigg),
\end{aligned}\end{equation*}
where $c_j=a_j^2/2$. We will prove the following more general assertion:
\begin{align}\label{mixing_formula}
    \mathbb{E}\bigg( &\exp \bigg \{ \sum_{j=1}^m b_j \tSp(z_j) -\sum_{j=m+1}^M b_j \big(\tSp(z_j+1+n)-\tSp(z_j+n)  \big) \bigg \}  \bigg) \\ &\notag\xrightarrow[n\to\infty]{} \mathbb{E}\bigg( \exp \bigg \{ \sum_{j=1}^m b_j \tSp(z_j)  \bigg \}  \bigg) \mathbb{E}\bigg( \exp \bigg \{ -\sum_{j=m+1}^M b_j \big(\tSp(z_j+1)-\tSp(z_j)  \big) \bigg \}  \bigg),
\end{align}
where $m,M\in\mathbb{N}$ with $1\leq m \leq M$, $0\leq z_1\leq \ldots \leq z_M$, $b_j\in \mathbb{R}$ for $j=1,\ldots,m$ and
$b_j\in \mathbb{R}_+$ for $j=m+1,\ldots,M$.

The proof follows the same lines as Lemma \ref{lemma2}. To show \eqref{mixing_formula} we will use induction on the parameter $M$:

\noindent
Step I. For $M=1$ also $m=1$ and the convergence trivially holds.

\noindent
Step II. Assume that the formula \eqref{mixing_formula} holds for $M-1$. We will
show that it holds also for $M$. We apply \eqref{recurrence2} to get
\begin{align*}
    \mathbb{E}&\bigg( \exp \bigg \{ \sum_{j=1}^m b_j \tSp(z_j) -\sum_{j=m+1}^M b_j \big(\tSp(z_j+1+n)-\tSp(z_j+n)  \big) \bigg \}  \bigg)\\
    &=\mathbb{E}\bigg( \exp \bigg \{ \sum_{j=2}^m b_j \tSp(z_j) -\sum_{j=m+1}^M b_j \big(\tSp(z_j+1+n)-\tSp(z_j+n)  \big) \bigg \}  \bigg)\\
    &\quad\mbox{}+ \frac{b_1}{\sum_{j=1}^m b_j} \int_0^{z_1} \mathbb{E}\bigg( \exp\bigg\{ \sum_{j=2}^m b_j \Sp(z_j-x) \\
    &\quad\mbox{}-\sum_{j=m+1}^M b_j \big(\Sp(z_j+1+n-x)-\Sp(z_j+n-x) \big) \bigg\} \bigg) d_x\mathbb{E}\bigg( \exp \bigg\{ \tilde \Sp(x) \sum_{j=1}^m b_j  \bigg\} \bigg)
\end{align*}
Now, applying the induction assumption for the first summand, Lemma \ref{lemma2} for the expression inside the integral,
and the dominated convergence theorem we obtain that the above formula converges to
\begin{gather*}
    \mathbb{E}\bigg( \exp \bigg \{ \sum_{j=2}^m b_j \tSp(z_j)\bigg \}  \bigg) \mathbb{E}\bigg( \exp \bigg \{  -\sum_{j=m+1}^M b_j \big(\tSp(z_j+1)-\tSp(z_j)  \big) \bigg \}  \bigg) + \\
    \mbox{}+ \mathbb{E}\bigg( \exp \bigg \{  -\sum_{j=m+1}^M b_j \big(\tSp(z_j+1)-\tSp(z_j)  \big) \bigg \}  \bigg)\times \\
    \mbox{}\times\frac{b_1}{\sum_{j=1}^m b_j} \int_0^{z_1} \mathbb{E}\bigg( \exp\bigg\{ \sum_{j=2}^m b_j \Sp(z_j-x) \bigg\} \bigg) d_x\mathbb{E}\bigg( \exp \bigg\{ \tSp(x) \sum_{j=1}^m b_j  \bigg\} \bigg)
\end{gather*}
which is, by \eqref{recurrence2}, the same as
$$
    \mathbb{E}\bigg( \exp \bigg \{ \sum_{j=1}^m b_j \tSp(z_j)  \bigg \}  \bigg) \mathbb{E}\bigg( \exp \bigg \{ -\sum_{j=m+1}^M b_j \big(\tSp(z_j+1)-\tSp(z_j)  \big) \bigg \}  \bigg).
$$
This shows that $\tYp(n)$ is mixing and therefore also ergodic.
\end{proof}
\begin{rem}
    We point out that Theorem \ref{Th_mixing} applies only if $\mathbb{E}(\Tp(1))<\infty$. For $\mathbb{E}(\Tp(1))=\infty$ one cannot construct the stationary-increment modification of $\Sp$ as in \eqref{stationary_subordinator}. As a consequence, the so-called \emph{weak ergodicity breaking} is observed, see \cite{Eli}.
\end{rem}

\begin{ack}
    M.M.\ is grateful to the Institut f\"ur Mathematische Stochastik of TU Dresden of its hospitality  during his Alexander von Humboldt fellowship at TU Dresden.The research of M.M.\ was partially supported by the Polish Ministry of Science and Higher Education grant No.~N N201 417639.
\end{ack}

\bibliographystyle{amsplain}

\end{document}